\documentclass[12pt,a4]{amsart}
\usepackage[a4paper, left=28mm, right=28mm, top=28mm, bottom=34mm]{geometry}
\usepackage[all]{xy}
\usepackage{amsmath}
\usepackage{amssymb}
\usepackage{amsthm}
\usepackage{mathrsfs}
\numberwithin{equation}{section}
\theoremstyle{definition}
\newtheorem{thm}{Theorem}[section]
\newtheorem{lem}[thm]{Lemma}
\newtheorem{cor}[thm]{Corollary}
\newtheorem{prop}[thm]{Proposition}

\theoremstyle{definition}

\newtheorem{rem}[thm]{Remark}

\newtheorem{defn}[thm]{Definition}

\def\Q{{\mathbb Q}}

\def\Z{{\mathbb Z}}

\def\O{{\mathscr O}}
\def\P{{\mathbb P}}

\def\Br{\mathop{\mathrm{Br}}\nolimits}

\def\Gal{\mathop{\mathrm{Gal}}\nolimits}

\def\Jac{\mathop{\mathrm{Jac}}\nolimits}

\def\GL{\mathop{\mathrm{GL}}\nolimits}

\def\Pic{\mathop{\mathrm{Pic}}\nolimits}

\def\deg{\mathop{\text{\rm deg}}\nolimits}

\def\chara{\mathop{\mathrm{char}}}

\newcommand{\transp}[1]{{}^{t}\!{#1}}

\begin{document}
\title[Symmetric determinantal representations of the Fermat curves]
{On the symmetric determinantal representations of the Fermat curves of prime degree}
\author{Yasuhiro Ishitsuka}
\address{Department of Mathematics, Faculty of Science, Kyoto University, Kyoto 606-8502, Japan}
\email{yasu-ishi@math.kyoto-u.ac.jp}
\author{Tetsushi Ito}
\address{Department of Mathematics, Faculty of Science, Kyoto University, Kyoto 606-8502, Japan}
\email{tetsushi@math.kyoto-u.ac.jp}

\date{\today}


\maketitle

\begin{abstract}
We prove that the defining equations of the Fermat curves of prime degree
cannot be written as the determinant of symmetric matrices with entries in linear forms
in three variables with rational coefficients.
In the proof, we use a relation between symmetric matrices with entries in linear forms
and non-effective theta characteristics on smooth plane curves.
We also use some results of Gross-Rohrlich on the rational torsion points
on the Jacobian varieties of the Fermat curves of prime degree.
\end{abstract}


\section{Introduction}

Let us consider the following Diophantine problem motivated by
the Arithmetic Invariant Theory of
symmetrized $3 \times n \times n$ boxes (\cite{BhargavaGrossWang}, \cite{Ho}):
for a smooth plane curve $C \subset \P^2_K$ of degree $n \geq 1$ defined over a field $K$,
does there exist a triple of symmetric matrices $(M_0,M_1,M_2)$ of size $n$
with entries in $K$ such that $C$ is defined by the equation of the form
\[ \det\big( X_0 M_0 + X_1 M_1 + X_2 M_2 \big) = 0 \ ? \]
If $C$ is defined by the above equation,
we say $C$ admits a {\em symmetric determinantal representation} over $K$.
We say two symmetric determinantal representations of $C$ defined by
triples $(M_0,M_1,M_2)$,\ $(M'_0,M'_1,M'_2)$ are {\em equivalent}
if there are $P \in \GL_n(K)$ and $a \in K^{\times}$
with $M_i' = a \, \transp{P} M_i P$ for $i = 0,1,2$,
where $\transp{P}$ is the transpose of the matrix $P$.
Note that equivalent triples give the same plane curve because
we have
\begin{align*}
  &\hspace*{-1.2cm} \det\big( X_0 M'_0 + X_1 M'_1 + X_2 M'_2 \big) \\
  &= \det\big(X_0 (a \, \transp{P} M_0 P) + X_1 (a \, \transp{P} M_1 P) + X_2 (a \, \transp{P} M_2 P) \big) \\
  &= a^n \cdot (\det P)^2 \cdot \det\big( X_0 M_0 + X_1 M_1 + X_2 M_2 \big).
\end{align*}

Finding symmetric determinantal representations of plane curves
is a classical problem in algebraic geometry.
For the history of this problem and known results,
see \cite{Dixon}, \cite{Edge1}, \cite{CookThomas}, \cite{Vinnikov}, \cite{BeauvilleDeterminantal},
\cite[Ch 4]{Dolgachev}.
When $K$ is algebraically closed of characteristic zero,
all plane curves (including singular ones) admit symmetric determinantal representations
(\cite[Remark 4.4]{BeauvilleDeterminantal}).
When $K$ is not algebraically closed,
many plane curves do not admit symmetric determinantal representations over $K$.
In \cite{IshitsukaIto1}, \cite{IshitsukaIto2},
we studied the local-global principle
for the existence of symmetric determinantal representations over global fields.

In this paper, we study symmetric determinantal representations of
the Fermat curves of prime degree and the Klein quartic over
the field $\Q$ of rational numbers.

We prove that the Fermat curves of prime degree
do {\em not} admit symmetric determinantal representations over $\Q$.

\begin{thm}
\label{MainTheorem}
Let $p \geq 2$ be a prime number.
The Fermat curve of degree $p$
\[ F_p := \big( X_0^p + X_1^p + X_2^p = 0 \big) \subset \P^2_{\Q} \]
does not admit a symmetric determinantal representation over $\Q$.
\end{thm}

\begin{rem}
Theorem \ref{MainTheorem} can be rephrased in concrete terms as follows:
there do {\em not} exist symmetric matrices $M_0,M_1,M_2$ of size $p$
with entries in $\Q$
and $a \in \Q^{\times}$ satisfying
\[ X_0^p + X_1^p + X_2^p = a \cdot \det\big( X_0 M_0 + X_1 M_1 + X_2 M_2 \big). \]
Since it is a Diophantine problem with $3 p^2 + 1$ variables,
it seems difficult to prove the non-existence of solutions in $\Q$ directly.
We shall prove Theorem \ref{MainTheorem} using the methods and the results from algebraic geometry.
\end{rem}

\begin{rem}
The Fermat curves are sometimes defined by the equation $X_0^p + X_1^p = X_2^p$
instead of $X_0^p + X_1^p + X_2^p = 0$.
There is no essential difference when $p$ is odd.
But there is a difference when $p = 2$.
In fact, the smooth conic $(X_0^2 + X_1^2 = X_2^2)$ admits
a symmetric determinantal representation over $\Q$.
(See Remark \ref{RemarkConicWithRationalPoints}.)
\end{rem}

The strategy of the proof of Theorem \ref{MainTheorem} is as follows.
The case of $p = 2$ is easy and treated separately.
Let $p \geq 3$ be an odd prime number.
For a smooth plane curve $C \subset \P^2_{\Q}$ over $\Q$,
triples of symmetric matrices giving rise to
symmetric determinantal representations correspond to certain line bundles on $C$ called
{\em non-effective theta characteristics}
(Proposition \ref{Proposition:ExistenceCriterion}).
Therefore, in order to prove Theorem \ref{MainTheorem},
we have to prove the non-existence of non-effective theta characteristics on $F_p$ over $\Q$.
If $\mathscr{L}$ is a non-effective theta characteristic on $F_p$ over $\Q$,
the line bundle $\mathscr{L} \otimes \O_{F_p}\big( (-p+3)/2 \big)$ gives
a non-trivial $\Q$-rational $2$-torsion point on the Jacobian variety $\Jac(F_p)$.
For an integer $s$ with $1 \leq s \leq p-2$,
let $C_s$ be the projective smooth model of the affine curve
\[ V^p = U (1 - U)^s. \]
Gross-Rohrlich calculated the $\Q$-rational torsion points on $\Jac(C_s)$ (\cite{GrossRohrlich}).
There is an isogeny
\[ \prod_{1 \leq s \leq p-2} \Jac(C_s) \longrightarrow \Jac(F_p) \]
defined over $\Q$ whose degree is a power of $p$.
We can calculate the $\Q$-rational $2$-torsion points on $\Jac(F_p)$
(Corollary \ref{Corollary:GrossRohrlich}).
When $p \neq 7$,
we have the non-existence of non-effective theta characteristics on $F_p$ over $\Q$.
The case of $p=7$ requires a little care because there are three non-trivial
$\Q$-rational $2$-torsion points on $\Jac(F_7)$.
In fact, we have
\[ \Jac(F_7)[2](\Q) \cong (\Z/2\Z)^2. \]
Fortunately, by explicit calculation,
we can prove these $2$-torsion points correspond to {\em effective} theta characteristics on $F_7$
over $\Q$.
Hence non-effective theta characteristics on $F_7$ over $\Q$ do not exist.
(The calculation of $\Q$-rational $p$-torsion points on $\Jac(F_p)$
is more subtle (\cite{Tzermias1}, \cite{Tzermias2}, \cite{Tzermias3}, \cite{Tzermias4}).
We do not need these results in this paper.)

Along the proof of Theorem \ref{MainTheorem},
we also study symmetric determinantal representations of the {\em Klein quartic}
\[ C_{\mathrm{Kl}} := \big( X_0^3 X_1 + X_1^3 X_2 + X_2^3 X_0 = 0 \big) \subset \P^2_{\Q} \]
over $\Q$.
It was already known to Klein that $C_{\mathrm{Kl}}$
admits a symmetric determinantal representation over $\Q$.
In fact, it is easy to confirm the following equality (cf.\ \cite[p.\,161]{Edge2}):
\begin{equation}
\label{Equation:KleinQuartic}
X_0^3 X_1 + X_1^3 X_2 + X_2^3 X_0 = 
- \det
\begin{pmatrix}
X_0 & 0 & 0 & -X_1 \\
0 & X_1 & 0 & -X_2 \\
0 & 0 & X_2 & -X_0 \\
-X_1 & -X_2 & -X_0 & 0
\end{pmatrix}
\end{equation}
We prove that the above expression gives a {\em unique equivalence class} of
symmetric determinantal representations of $C_{\mathrm{Kl}}$ over $\Q$.


\begin{thm}
\label{MainTheoremKlein}
There is a unique equivalence class of symmetric determinantal representations of
the Klein quartic $C_{\mathrm{Kl}}$ over $\Q$.
\end{thm}

The proof of Theorem \ref{MainTheoremKlein} is as follows.
It is classically known that $C_{\mathrm{Kl}}$
is birational over $\Q$ to the curve $C_2$ (or $C_4$) for $p = 7$
studied by Gross-Rohrlich (\cite[p.\,67]{Elkies}).
Using the results of Gross-Rohrlich,
we see that
\[ \Jac(C_{\mathrm{Kl}})[2](\Q) \cong \Z/2\Z. \]
Hence there are exactly two theta characteristics on $C_{\mathrm{Kl}}$ over $\Q$,
up to isomorphism.
There is at least one {\em effective} theta characteristic on $C_{\mathrm{Kl}}$
because the line $(X_0 + X_1 + X_2 = 0)$ is a bitangent to $C_{\mathrm{Kl}}$
over $\Q$, which corresponds to an {\em effective} theta characteristic
on $C_{\mathrm{Kl}}$ over $\Q$ (\cite[Ch.\ 6]{Dolgachev}).
Hence there is exactly one non-effective theta characteristic on $C_{\mathrm{Kl}}$
over $\Q$, up to isomorphism.

\begin{rem}
There is a long history on the study of the Klein quartic.
There is no surprise if some geometers have expected (or presumed)
that Theorem \ref{MainTheoremKlein} could be true.
Note that a rigorous proof of Theorem \ref{MainTheoremKlein}
requires the study of the field of definition of
equivalence classes of symmetric determinantal representations,
which is a slightly delicate arithmetic problem.
(See Corollary \ref{Corollary:ThetaCharacteristicOrder} (4).
See also \cite{Ho}, \cite{Ishitsuka}.)
These days, we become more interested in the arithmetic 
properties of linear orbits related to symmetric determinantal representations
thanks to the recent developments of Arithmetic Invariant Theory
(\cite{Ho}, \cite{BhargavaGrossWang}).
\end{rem}

In Section \ref{Section:SymmetricDeterminantalRepresentations},
we recall a relation between symmetric determinantal representations and
non-effective theta characteristics.
Theorem \ref{MainTheorem} for $p = 2$ is a consequence of the fact
that the conic $(X_0^2 + X_1^2 + X_2^2 = 0)$ has no $\Q$-rational points.
In Section \ref{Section:GrossRohrlich},
we recall some results of Gross-Rohrlich,
and prove Theorem \ref{MainTheorem} for $p \neq 2,7$.
The proof of Theorem \ref{MainTheorem} for $p = 7$
is given in Section \ref{Section:p=7}.
Finally, the proof of Theorem \ref{MainTheoremKlein}
is given in Section \ref{Section:ProofKlein}.

\section{Theta characteristics and symmetric determinantal representations}
\label{Section:SymmetricDeterminantalRepresentations}

Let $K$ be a field,
and $C \subset \P^2_K$ a smooth plane curve of degree $n \geq 1$.
The genus of $C$ is equal to $g(C) := (n-1)(n-2)/2$.

\begin{defn}[\cite{MumfordTheta}]
\label{DefinitionTheta}
\begin{enumerate}
\item A {\em theta characteristic} on $C$ is a line bundle $\mathscr{L}$ on $C$
satisfying $\mathscr{L} \otimes \mathscr{L} \cong \Omega^1_{C}$,
where $\Omega^1_{C}$ is the canonical sheaf on $C$.
\item A theta characteristic $\mathscr{L}$ on $C$
is {\em effective} (resp.\ {\em non-effective}) 
if $H^0(C,\mathscr{L}) \neq 0$ (resp.\ $H^0(C,\mathscr{L}) = 0$).
\end{enumerate}
\end{defn}

\begin{prop}
\label{Proposition:ExistenceCriterion}
There is a bijection between
the set of isomorphism classes of non-effective theta characteristics on $C$
and the set of equivalence classes
of symmetric determinantal representations of $C$ over $K$.
\end{prop}

\begin{proof}
This proposition is well-known when $\chara K \neq 2$
(\cite[Proposition 6.23]{BeauvillePrym}, \cite[Proposition 4.2]{BeauvilleDeterminantal},
\cite[Ch 4]{Dolgachev}, \cite[Theorem 4.12]{Ho}).
It is not difficult to modify the arguments in \cite{BeauvilleDeterminantal}
to cover the case of characteristic two (\cite[Remark 2.2]{BeauvilleDeterminantal}).
For a proof of this proposition which works over arbitrary fields,
see also \cite{Ishitsuka}.
\end{proof}

In order to study the field of definition of
equivalence classes of symmetric determinantal representations,
we use the Picard scheme and the Jacobian variety of the smooth plane curve $C$
(\cite{BoschLuetkebohmertRaynaud}, \cite{KleimanPicard}).
Let us recall their basic properties.
The {\em Picard group}
\[ \Pic(C) := H^1(C,\O_C^{\times}) \]
is the group of isomorphism classes of line bundles on $C$.
Let $\Pic_{C/K}$ be the {\em Picard scheme} of $C$
representing the relative Picard functor (\cite[Theorem 3 in \S 8.2]{BoschLuetkebohmertRaynaud}).
We have the following exact sequence
\[
\xymatrix{
0 \ar[r] & \Pic(C)
  \ar[r] & \Pic_{C/K}(K)
  \ar[r] & \Br(K),
}
\]
where $\Br(K)$ is the Brauer group of $K$.

The equality $\Pic(C) = \Pic_{C/K}(K)$ holds if $C$ has a $K$-rational point
(\cite[Proposition 4 in \S 8.1]{BoschLuetkebohmertRaynaud}).
The identity component of $\Pic_{C/K}$
is denoted by $\Jac(C)$ called the {\em Jacobian variety} of $C$.
It is known that $\Jac(C)$ is an abelian variety over $K$
of dimension $g(C) = (n-1)(n-2)/2$
(\cite[Proposition 3 in \S 9.2]{BoschLuetkebohmertRaynaud}).

Let $\Pic(C)[2]$ be the group of $2$-torsion points on $\Pic(C)$,
which is the group of isomorphism classes of line bundles $\mathscr{L}$
with $\mathscr{L} \otimes \mathscr{L} \cong \O_C$.
The group of $K$-rational $2$-torsion points on $\Jac(C)$
is denoted by $\Jac(C)[2](K)$.
It is always true that $\Pic(C)[2]$ is a subgroup of $\Jac(C)[2](K)$.
These two groups are not necessarily equal if $C$ has no $K$-rational points.

Using Proposition \ref{Proposition:ExistenceCriterion}
and the following corollary,
we give an upper bound of the number of equivalence classes of
symmetric determinantal representations.

\begin{cor}
\label{Corollary:ThetaCharacteristicOrder}
\begin{enumerate}
\item The number of isomorphism classes of theta characteristics on $C$
is less than or equal to the order of $\Jac(C)[2](K)$.
\item If $C$ has a $K$-rational point,
the number of isomorphism classes of theta characteristics on $C$
is zero or equal to the order of $\Jac(C)[2](K)$.
\item The number of equivalence classes of symmetric determinantal representations of $C$
over $K$ is finite.
\item Let $L/K$ be an extension of fields.
Two symmetric determinantal representations of $C$ over $K$ are equivalent over $K$
if and only if they are equivalent over $L$.
\end{enumerate}
\end{cor}

\begin{proof}
(1) \ Assume that there is a theta characteristic $\mathscr{L}$ on $C$.
The other theta characteristics on $C$
are of the form $\mathscr{L} \otimes \mathscr{L}'$,
where $\mathscr{L}'$ is a line bundle on $C$ with $\mathscr{L}' \otimes \mathscr{L}' \cong \O_C$.
Hence the number of isomorphism classes of
theta characteristics on $C$ is equal to the order of $\Pic(C)[2]$,
which is less than or equal to the order of $\Jac(C)[2](K)$
because $\Pic(C)[2]$ is a subgroup of $\Jac(C)[2](K)$.

\vspace{0.1in}

\noindent
(2) \ If $C$ has a $K$-rational point, we have $\Pic(C)[2] = \Jac(C)[2](K)$,
and the assertion (2) follows.

\vspace{0.1in}

\noindent
(3) \ Since $\Jac(C)[2](K)$ is a finite group, the assertion (3) follows.

\vspace{0.1in}

\noindent
(4) \ Since $(\Pic_{C/K}) \otimes_K L = \Pic_{C \otimes_K L/L}$,
we have $\Pic_{C/K}(L) = \Pic_{C \otimes_K L/L}(L)$.
The map $\Pic_{C/K}(K) \longrightarrow \Pic_{C \otimes_K L/L}(L)$ is injective.
Hence $\Pic(C) \longrightarrow \Pic(C \otimes_K L)$ is also injective, and
the assertion (4) follows.
\end{proof}

\begin{rem}
If we define the notion of theta characteristics on singular plane curves
as in \cite{Piontkowski},
Proposition \ref{Proposition:ExistenceCriterion} can be generalized to
the case of singular plane curves.
There is a natural bijection between equivalence classes of symmetric determinantal representations
of $C$ over $K$
and isomorphism classes of non-effective theta characteristics $\mathscr{L}$ on $C$
equipped with a certain isomorphism between $\mathscr{L}$ and the dual of it
(\cite{Ishitsuka}).
However, when $C$ has singularities,
Corollary \ref{Corollary:ThetaCharacteristicOrder} (3),(4) are not true in general.
(See \cite{Ishitsuka} for details.)
\end{rem}

\begin{cor}
\label{Corollary:OddDegreeNonExistence}
Assume that $n \geq 3$ and $n$ is {\em odd}.
If $\Jac(C)[2](K) = 0$, the smooth plane curve
$C$ does not admit a symmetric determinantal representation over $K$.
\end{cor}

\begin{proof}
By Corollary \ref{Corollary:ThetaCharacteristicOrder} (1),
there is {\em at most one} isomorphism class of theta characteristics on $C$.
By Proposition \ref{Proposition:ExistenceCriterion},
we have only to show that there is an {\em effective} theta characteristic on $C$.
Since $C \subset \P^2_K$ is a smooth plane curve of degree $n$,
the canonical sheaf $\Omega^1_{C}$ is isomorphic to
the restriction
\[ \O_{C}(n-3) := \O_{\P^2_K}(n-3)|_{C} \]
(\cite[II., 8.20.3]{Hartshorne}, \cite[Exercise 6.4.11]{Liu}).
Hence
\[ \O_{C}\big((n-3)/2\big) := \O_{\P^2_K}\big((n-3)/2\big)|_{C} \]
is a theta characteristic on $C$.
It is effective because homogeneous polynomials in $X_0,X_1,X_2$ of degree $(n-3)/2$
give global sections of it.
\end{proof}

In the following proposition, we consider symmetric determinantal representations
of smooth conics.
(See also \cite[Proposition 5.1]{IshitsukaIto2})

\begin{prop}
\label{Proposition:Conic}
Assume that $n = 2$.
For a smooth conic $C \subset \P^2_K$ over $K$, the following are equivalent.
\begin{enumerate}
\item $C$ is isomorphic to $\P^1_K$ over $K$.
\item $C$ has a $K$-rational point.
\item $C$ has a line bundle of odd degree over $K$.
\item $C$ admits a symmetric determinantal representation over $K$.
\end{enumerate}
If the above conditions are satisfied,
there is a unique equivalence class of 
symmetric determinantal representations of $C$ over $K$.
\end{prop}

\begin{proof}
$(3) \Rightarrow (4)$ \ Since the smooth conic $C$ is a projective smooth curve of genus 0,
we have $\Jac(C) = 0$ and $\Pic(C) \subset \Pic_{C/K}(K) \cong \Z$.
The isomorphism $\Pic_{C/K}(K) \overset{\cong}{\longrightarrow} \Z$ is given by the degree of line bundles.
Since $\deg \Omega^1_{C} = -2$,
$\Pic(C)$ is a subgroup of $\Pic_{C/K}(K)$ of index less than or equal to $2$.
If $C$ has a line bundle of odd degree over $K$,
we have $\Pic(C) = \Pic_{C/K}(K)$.
There is a line bundle $\mathscr{L}$ of degree $-1$,
which is a non-effective theta characteristic on $C$ over $K$.
By Proposition \ref{Proposition:ExistenceCriterion},
$C$ admits a symmetric determinantal representation over $K$.

\vspace{0.1in}

\noindent
$(4) \Rightarrow (3)$ \ If $C$ admits a symmetric determinantal representation over $K$,
there is a non-effective theta characteristic $\mathscr{L}$ on $C$ 
by Proposition \ref{Proposition:ExistenceCriterion}.
We see that $\deg \mathscr{L} = -1$ is odd.

\vspace{0.1in}

\noindent
$(1) \Leftrightarrow (2) \Leftrightarrow (3)$ \ These implications are well-known.
We briefly recall the proof.
The implications (1) $\Rightarrow$ (2) $\Rightarrow$ (3) are obvious.
Assume that $C$ has a line bundle of odd degree.
We have $\Pic(C) = \Pic_{C/K}(K) \cong \Z$,
and there is a divisor $D$ on $C$ of degree 1.
Since the complete linear system $|D|$ is one-dimensional and very ample,
$C$ is isomorphic to $\P^1_K$ over $K$ (\cite[IV., 3.3.1]{Hartshorne}, \cite[Proposition 7.4.1]{Liu}).

\vspace{0.1in}

\noindent
The last assertion follows from Corollary \ref{Corollary:ThetaCharacteristicOrder} (2).
\end{proof}

\begin{proof}[\bf \em Proof of Theorem \ref{MainTheorem} (for $p = 2$)]
The smooth conic $(X_0^2 + X_1^2 + X_2^2 = 0)$ has no $\Q$-rational points.
It does not admit a symmetric determinantal representation over $\Q$
by Proposition \ref{Proposition:Conic}.
\end{proof}

\begin{rem}
\label{RemarkConicWithRationalPoints}
The Fermat curves are sometimes defined by the equation
\[ X_0^p + X_1^p = X_2^p \]
instead of $X_0^p + X_1^p + X_2^p = 0$.
There is no essential difference when $p$ is odd.
However, when $p = 2$, the conic $(X_0^2 + X_1^2 = X_2^2)$
is not isomorphic over $\Q$ to the conic $(X_0^2 + X_1^2 + X_2^2 = 0)$.
Since the conic $(X_0^2 + X_1^2 = X_2^2)$ has a $\Q$-rational point such as $(1,0,1)$,
there is a unique equivalence class of symmetric determinantal representations of it
over $\Q$ by Proposition \ref{Proposition:Conic}.
For example, we have
\[
X_0^2 + X_1^2 - X_2^2 =
- \det
\begin{pmatrix}
X_0 + X_2 & X_1 \\ X_1 & -X_0+X_2
\end{pmatrix}.
\]
\end{rem}


\section{Results of Gross-Rohrlich and
the non-existence of symmetric determinantal representations of the Fermat curve of degree $p \neq 7$}
\label{Section:GrossRohrlich}

We recall some results of Gross-Rohrlich on the $\Q$-rational torsion points
on the Jacobian varieties of the Fermat curve of prime degree (\cite{GrossRohrlich}).

Let $p \geq 3$ be an odd prime number,
and
\[ F_p : = \big( X_0^p + X_1^p + X_2^p = 0 \big) \subset \P^2_{\Q} \]
the Fermat curve of degree $p$ over $\Q$.
For an integer $s$ with $1 \leq s \leq p-2$,
let $C_s$ be the projective smooth model of the affine curve
\[ V^p = U (1 - U)^s. \]

\begin{thm}[Theorem 1.1 in \cite{GrossRohrlich}]
\label{Theorem:GrossRohrlich}
Let $\Jac(C_s)(\Q)_{\text{\rm tors}}$ be the group of $\Q$-rational torsion points on $\Jac(C_s)$.
Then we have
\[
\Jac(C_s)(\Q)_{\text{\rm tors}} \cong
\begin{cases}
\Z/p\Z & p \neq 7 \ \text{or} \ (p,s) = (7,1),(7,3),(7,5) \\
\Z/7\Z \times \Z/2\Z & (p,s) = (7,2),(7,4).
\end{cases}
\]
\end{thm}

Let us calculate the $\Q$-rational prime-to-$p$ torsion points on $\Jac(F_p)$
using Theorem \ref{Theorem:GrossRohrlich}.
There is a finite morphism
\[ \varphi_s \colon F_p \longrightarrow C_s \]
of degree $p$ defined by
\[
\varphi_s(X_0,X_1,X_2) = \big( -(X_0/X_2)^p,\  (-1)^{s+1} X_0 X_1^s/X_2^{s+1} \,\big).
\]

\begin{rem}
The above expression may be slightly confusing.
The coordinates $(X_0,X_1,X_2)$ of $F_p$ are homogeneous coordinates in
the projective plane $\P^2_{\Q}$,
whereas the coordinates $(U,V)$ of $C_s$ are affine coordinates.
There is a difference of signs from \cite[p.\,207]{GrossRohrlich}
because the defining equation of the Fermat curve in \cite{GrossRohrlich}
is $X^p + Y^p = 1$.
\end{rem}

The morphism $\varphi_s$ induces morphisms between Jacobian varieties
\[ \varphi_{s \ast} \colon \Jac(F_p) \longrightarrow \Jac(C_s), \qquad
   \varphi_s^{\ast} \colon \Jac(C_s) \longrightarrow \Jac(F_p) \]
corresponding to the pushfoward and the pullback of divisor classes.
Let $\varphi_{\ast},\varphi^{\ast}$ be the product of $\varphi_{s \ast},\varphi_s^{\ast} \ (1 \leq s \leq p-2)$,
respectively.
Faddeev proved that the product $\varphi_{\ast}$ induces an isogeny (\cite{Faddeev}):
\[
\varphi_{\ast} \colon \Jac(F_p) \ \longrightarrow \ \prod_{1 \leq s \leq p-2} \Jac(C_s).
\]
Hence $\varphi^{\ast}$ is an isogeny from $\prod_{1 \leq s \leq p-2} \Jac(C_s)$ to $\Jac(F_p)$.

\begin{lem}
\label{Lemma:FermatJacobianIsogeny}
The composite $\varphi^{\ast} \circ \varphi_{\ast} \colon \Jac(F_p) \longrightarrow\Jac(F_p)$
is equal to the multiplication-by-$p$ isogeny.
\end{lem}

\begin{proof}
This result is presumably well-known (\cite[p.\ 2]{Tzermias1}).
Since it is not stated explicitly in \cite{GrossRohrlich},
we briefly sketch how to deduce it from the results in \cite{Rohrlich}, \cite{GrossRohrlich}.
It is enough to prove Lemma \ref{Lemma:FermatJacobianIsogeny}
over $\overline{\Q}$.
So we shall work over $\overline{\Q}$ in the following argument.
Let $\zeta \in \overline{\Q}$ be a primitive $p$-th root of unity.
Define the automorphisms $A,B$ of $F_p$ by
\[
  A(X_0,X_1,X_2) := (\zeta X_0,X_1,X_2), \quad
  B(X_0,X_1,X_2) := (X_0,\zeta X_1,X_2).
\]
Since $\varphi_s \colon F_p \longrightarrow C_s$ is a Galois covering
of degree $p$ whose Galois group is generated by $A^{-s} B$,
we have
\[ \varphi^{\ast} \circ \varphi_{\ast} \ = \sum_{s=1}^{p-2} \sum_{j=0}^{p-1} \,(A^{-s} B)^j \]
on $\Jac(F_p)$ (\cite[p.\,208]{GrossRohrlich}).
We have
\[
\varphi^{\ast} \circ \varphi_{\ast} - p
= \bigg( \sum_{j=0}^{p-1} A^j \bigg) \bigg( \sum_{k=0}^{p-1} B^k \bigg)
- \sum_{j=0}^{p-1} A^j
- \sum_{k=0}^{p-1} B^k
- \sum_{j=0}^{p-1} (AB)^j.
\]
on $\Jac(F_p)$.
Therefore, we have $\varphi^{\ast} \circ \varphi_{\ast} - p = 0$ on $\Jac(F_p)$
because the operators
$\sum_{j=0}^{p-1} A^j$,\ $\sum_{k=0}^{p-1} B^k$,\ $\sum_{j=0}^{p-1} (AB)^j$
annihilate $\Jac(F_p)$ by \cite[Remark in p.\,121]{Rohrlich}.
\end{proof}

By Lemma \ref{Lemma:FermatJacobianIsogeny}, we have an isomorphism
\[ \Jac(F_p)(\Q)_{p'\text{\rm -tors}}
   \ \cong \ \prod_{1 \leq s \leq p-2} \Jac(C_s)(\Q)_{p'\text{\rm -tors}}, \]
where
$\Jac(F_p)(\Q)_{p'\text{\rm -tors}}$
(resp.\ $\Jac(C_s)(\Q)_{p'\text{\rm -tors}}$)
denotes the group of $\Q$-rational torsion points on
$\Jac(F_p)$ (resp.\ $\Jac(C_s)$) whose orders are prime to $p$.

\begin{cor}
\label{Corollary:GrossRohrlich}
\[ \Jac(F_p)(\Q)_{p'\text{\rm -tors}} \cong
   \begin{cases}
     0 & p \neq 7 \\
     (\Z/2\Z)^2 & p = 7
   \end{cases} \]
\end{cor}

\begin{proof}[\bf \em Proof of Theorem \ref{MainTheorem} (for $p \neq 2,7$)]
There does not exist non-trivial $\Q$-rational $2$-torsion points on 
$\Jac(F_p)$ by Corollary \ref{Corollary:GrossRohrlich}.
Hence $F_p$ does not admit a symmetric determinantal representation over $\Q$
by Corollary \ref{Corollary:OddDegreeNonExistence}.
\end{proof}

\section{The non-existence of symmetric determinantal representations of the Fermat curve of degree $7$}
\label{Section:p=7}

This case requires a little care because we have
\[ \Jac(F_7)[2](\Q) \cong (\Z/2\Z)^2 \]
by Corollary \ref{Corollary:GrossRohrlich}.
We shall prove these $2$-torsion points correspond to {\em effective} theta characteristics on $F_7$.

\begin{proof}[\bf \em Proof of Theorem \ref{MainTheorem} (for $p = 7$)]
We recall the results of Gross-Rohrlich on divisors on $F_7$ representing elements of
$\Jac(F_7)[2](\Q) \cong (\Z/2\Z)^2$.
(See \cite[p.\,209, (3)]{GrossRohrlich} for details.)

Let $\varepsilon \in \overline{\Q}$ be a primitive $14$-th root of unity,
and we put $\zeta := \varepsilon^2$.
Let $\eta \in \overline{\Q}$ be a primitive $6$-th root of unity.
Then $\eta$ is conjugate to $\eta^{-1}$ over $\Q$.
We put
\[ P := (\eta, \eta^{-1}, -1),\qquad
   Q := (\eta^{-1}, \eta, -1), \qquad
   R_j := (\varepsilon \zeta^j, 1, 0) \quad (0 \leq j \leq 6) \]
Define a divisor $D$ of degree $0$ on $F_7$ by
\[ D := \sum_{j=0}^6 \big(\, (A^3 B)^j(P) + (A^3 B)^j(Q) - 2 R_j \,\big), \]
where $A,B$ are the automorphisms of $F_7$ defined
in the proof of Lemma \ref{Lemma:FermatJacobianIsogeny}.
The divisor $D$ is invariant under the action of $\Gal(\overline{\Q}/\Q)$.
Hence the line bundle $\O_{F_7}(D)$ is defined over $\Q$.

It is straightforward to confirm that the divisor of the rational function
\[ f := X_2^{-4} \big( X_1^3 X_2 + X_0 X_2^3 - X_0^3 X_1 \big) \]
on $F_7$ is equal to $2D$. Hence $2D$ is a principal divisor, and
$D$ gives a $\Q$-rational $2$-torsion point
\[ [D] \in \Jac(F_7)[2](\Q). \]
Gross-Rohrlich proved that the divisor class $[D]$ is non-trivial,
and
\[ \Jac(F_7)[2](\Q) \ = \ \big\{ \, 0,\ [D],\ [\sigma(D)],\ [\sigma^2(D)] \, \big\}. \]
Here $\sigma \colon F_7 \longrightarrow F_7$ is the automorphism of order $3$ defined by
\[ \sigma(X_0,X_1,X_2) := (X_1,X_2,X_0). \]

Recall that $\O_{F_7}(2) := \O_{\P^2_{\Q}}(2)|_{F_7}$
is an {\em effective} theta characteristic on $F_7$.
(See the proof of Corollary \ref{Corollary:OddDegreeNonExistence}.)
Since $F_7$ has a $\Q$-rational point such as $(1, -1, 0)$,
there are exactly 4 isomorphism classes of theta characteristics on $F_7$
by Corollary \ref{Corollary:ThetaCharacteristicOrder} (2).
They are represented by the following line bundles
\[
\O_{F_7}(2),\quad
\O_{F_7}(2) \otimes \O_{F_7}(D),\quad
\O_{F_7}(2) \otimes \O_{F_7}\big( \sigma(D) \big),\quad
\O_{F_7}(2) \otimes \O_{F_7}\big( \sigma^2(D) \big).
\]

By Proposition \ref{Proposition:ExistenceCriterion},
we need to prove that all of these theta characteristics are effective.
We have already seen that $\O_{F_7}(2)$ is effective.
Since the automorphism $\sigma$ permutes
the three theta characteristics on $F_7$ except $\O_{F_7}(2)$,
it is enough to show that one of them is effective.

We shall show $\O_{F_7}(2) \otimes \O_{F_7}(D)$ is effective.
The line bundle $\O_{F_7}(2) \otimes \O_{F_7}(D)$ is effective
if and only if the divisor $D + 2 H$ is linearly equivalent
to an effective divisor,
where $H$ is a hyperplane section (i.e.\ $H$ is a divisor on $F_7$
cut out by a line in $\P^2_{\Q}$).
Consider the line at infinity $\ell_{\infty} \subset \P^2_{\Q}$ defined by $X_2 = 0$.
The line $\ell_{\infty}$ intersects with $F_7$ at the 7 points $R_0,R_1,\ldots,R_6$.
All of the intersection points have multiplicity one.
(Since $F_7 \subset \P^2_{\Q}$ is a plane curve of degree $7$,
there are exactly $7$ intersections counted with their multiplicities
by B\'ezout's theorem (\cite[I., 7.8]{Hartshorne}, \cite[Corollary 9.1.20]{Liu}).)
The divisor $\ell_{\infty} \cap F_7$ on $F_7$ cut out by $\ell_{\infty}$ is
equal to $\sum_{j=0}^{6} R_j$.
Since
\[ D + 2(\ell_{\infty} \cap F_7) = \sum_{j=0}^6 \big(\, (A^3 B)^j(P) + (A^3 B)^j(Q) \,\big) \]
is effective,
we conclude that $\O_{F_7}(2) \otimes \O_{F_7}(D)$ is an effective
theta characteristic on $F_7$.

The proof of Theorem \ref{MainTheorem} for $p=7$ is complete.
\end{proof}

\section{Symmetric determinantal representations of the Klein quartic over $\Q$}
\label{Section:ProofKlein}

The Klein quartic is a smooth plane curve of degree $4$ over $\Q$
defined by the equation
\[ C_{\mathrm{Kl}} := \big( X_0^3 X_1 + X_1^3 X_2 + X_2^3 X_0 = 0 \big) \subset \P^2_{\Q}. \]
In this section, we study symmetric determinantal representations of
$C_{\mathrm{Kl}}$ over $\Q$.
For an excellent exposition of the arithmetic and the geometry of the Klein quartic, see \cite{Elkies}.

The following arithmetic results on the Klein quartic must be well-known.

\begin{lem}
\label{Lemma:KleinQuartic}
\begin{enumerate}
\item $\Jac(C_{\mathrm{Kl}})[2](\Q) \cong \Z/2\Z$.
\item There is an effective theta characteristic on $C_{\mathrm{Kl}}$ defined over $\Q$.
\end{enumerate}
\end{lem}

\begin{proof}
(1) \ It seems possible to work with the defining equation of $C_{\mathrm{Kl}}$ directly.
Here we shall deduce it from the results of Gross-Rohrlich (\cite{GrossRohrlich}).
The Klein quartic $C_{\mathrm{Kl}}$ is isomorphic to
the projective smooth model $C_2$ of the affine curve $V^7 = U (1-U)^2$ (\cite[p.\,67]{Elkies}).
This can be seen as follows.
We have a rational map
\[ C_{\mathrm{Kl}} \dashrightarrow C_2,\qquad (a,b,c) \mapsto (s,t) := (-a^2 b/c^3,\,-b/c) \]
Since $t^7 = s(1-s)^2$ is satisfied, $(s,t)$ lies on the affine curve $V^7 = U (1-U)^2$.
Conversely, since $c/b = -1/t,\ a/b = t^2/(s-1)$,
the homogeneous coordinates of $(a,b,c)$ are recovered from $(s,t)$ when
$s \neq 1,\ t \neq 0$.
Hence $C_{\mathrm{Kl}}$ is birational to $C_2$ over $\Q$.
Since $C_{\mathrm{Kl}}$, $C_2$ are projective smooth curves over $\Q$,
they are isomorphic over $\Q$
(\cite[I., 6.12]{Hartshorne}, \cite[Proposition 7.3.13 (b)]{Liu}).
Hence we have $C_{\mathrm{Kl}} \cong C_2$, and
\[ \Jac(C_{\mathrm{Kl}})[2](\Q) \cong \Jac(C_2)[2](\Q) \cong \Z/2\Z \]
by Theorem \ref{Theorem:GrossRohrlich}.

\vspace{2mm}

\noindent
(2) \ Let $\zeta_3 \in \overline{\Q}$ be a primitive cube root of unity.
Consider $\overline{\Q}$-rational points
$P := (\zeta_3, \zeta_3^2, 1)$ and $Q := (\zeta_3^2, \zeta_3, 1)$ on $C_{\mathrm{Kl}}$.
The line
\[ X_0 + X_1 + X_2 = 0 \]
is a bitangent to $C_{\mathrm{Kl}}$ at $P,Q$.
Since $P + Q$ is a $\Gal(\overline{\Q}/\Q)$-invariant divisor on $C_{\mathrm{Kl}} \otimes_{\Q} \overline{\Q}$,
the line bundle $\O_{C_{\mathrm{Kl}}}(P+Q)$ is defined over $\Q$.
The canonical sheaf $\Omega^1_{C_{\mathrm{Kl}}}$ is isomorphic to
the restriction $\O_{C_{\mathrm{Kl}}}(1) := \O_{\P^2_{\Q}}(1)|_{C_{\mathrm{Kl}}}$.
Since the divisor $2P + 2Q$ on $C_{\mathrm{Kl}}$ is cut out by a line,
we have $\O_{C_{\mathrm{Kl}}}(2P+2Q) \cong \O_{C_{\mathrm{Kl}}}(1)$.
Hence $\O_{C_{\mathrm{Kl}}}(P+Q)$ is an effective theta characteristic on $C_{\mathrm{Kl}}$
defined over $\Q$.
\end{proof}

\begin{rem}
There are 28 bitangents to a smooth plane quartic defined over
an algebraically closed field of characteristic different from two.
There is a bijection between bitangents
and isomorphism classes of effective theta characteristics (\cite[Ch.\ 6]{Dolgachev}).
The defining equations of the 28 bitangents to
the Klein quartic $C_{\mathrm{Kl}} \otimes_{\Q} \overline{\Q}$
over $\overline{\Q}$ can be found in \cite[Proposition 9]{Shioda}.
In fact, Shioda calculated the defining equations over any algebraically
closed fields of characteristic different from $7$.
One can verify that $X_0 + X_1 + X_2 = 0$
is a unique bitangent defined over $\Q$.
\end{rem}

\begin{proof}[\bf \em Proof of Theorem \ref{MainTheoremKlein}]
Since $C_{\mathrm{Kl}}$ has a $\Q$-rational point such as $(1, 0, 0)$,
there are exactly two theta characteristics on $C_{\mathrm{Kl}}$, up to isomorphism,
by Corollary \ref{Corollary:ThetaCharacteristicOrder} (2) and Lemma \ref{Lemma:KleinQuartic}.
Since $C_{\mathrm{Kl}}$ admits a symmetric determinantal representation over $\Q$
(cf.\ the equation (\ref{Equation:KleinQuartic}) in Introduction),
there is a {\em non-effective} theta characteristic on $C_{\mathrm{Kl}}$
by Proposition \ref{Proposition:ExistenceCriterion}.
On the other hand, there is
an {\em effective} theta characteristic on $C_{\mathrm{Kl}}$ by Lemma \ref{Lemma:KleinQuartic} (2).
Therefore, there is a {\em unique} isomorphism class of
non-effective theta characteristics on $C_{\mathrm{Kl}}$ over $\Q$.
The assertion of Theorem \ref{MainTheoremKlein}
follows from Proposition \ref{Proposition:ExistenceCriterion}.
\end{proof}


\subsection*{Acknowledgements}

The work of the first author was supported by JSPS KAKENHI Grant Number 13J01450.
The work of the second author 
was supported by JSPS KAKENHI Grant Number 20674001 and 26800013.
The authors would like to thank an anonymous referee
for simplifying the proof of Lemma 3.3.

\end{document}